\documentclass[reqno]{amsart}

\usepackage[sorting=none, sorting=nyt, maxbibnames=99, backend=biber]{biblatex}

\renewbibmacro{in:}{}
\bibliography{references.bib}
\usepackage{amssymb}
\usepackage{calrsfs}
\usepackage{graphics,graphicx, mathrsfs}
\usepackage{enumerate}
\usepackage{enumitem}
\usepackage{url}
\usepackage[table]{xcolor}
\definecolor{vio}{rgb}{0.54, 0.17, 0.89}
\usepackage[ruled, lined, linesnumbered, longend]{algorithm2e}
\usepackage{hyperref}
\usepackage{titlesec}

\numberwithin{equation}{section}

\theoremstyle{remark}

\titleformat{\section}
  {\normalfont\large\bfseries\centering}{\thesection}{1em}{}

\titleformat{\subsection}
  {\normalfont\bfseries}{\thesubsection}{1em}{}

\newcommand{\abs}[1]{\left\lvert#1\right\rvert}

\def\reals{\hbox{\rm I\kern-.18em R}}
\def\complexes{\hbox{\rm C\kern-.43em
\vrule depth 0ex height 1.4ex width .05em\kern.41em}}
\def\field{\hbox{\rm I\kern-.18em F}} 

\newcommand\blfootnote[1]{%
  \begingroup
  \renewcommand\thefootnote{}\footnote{#1}%
  \addtocounter{footnote}{-1}%
  \endgroup
}

\newenvironment{section*}[2][A]{
  \section*{#2}
  \renewcommand\thesection{#1}
  \setcounter{theorem}{0}}{}

\allowdisplaybreaks
\numberwithin{equation}{section}
\usepackage{mathtools}
\usepackage{thmtools}
\usepackage{hyperref}

\theoremstyle{definition}

\newtheorem{Th}{Theorem}[section]
\newtheorem{Lemma}[Th]{Lemma}
\theoremstyle{definition}
\newtheorem{Cor}[Th]{Corollary}

\theoremstyle{definition}

\newtheorem*{Th!}{Theorem}
\newtheorem{Prop}[Th]{Proposition}

\newtheorem{Rem}[Th]{Remark}
\newtheorem{?}[Th]{Problem}

\newcommand{\bb}[1]{\mathbb{#1}}

\renewcommand{\bf}[1]{\textbf{#1}}

\renewcommand{\mod}[1]{ \ (\text{mod } #1 ) }

\newcommand{\fin}[1]{\bb{F}_{q^{#1}}}
\newcommand{\sect}{§}

\usepackage{dirtytalk}

\begin{document}

\title[Hybrid bounds for prime divisors]{Hybrid bounds for prime divisors}

\author{Gustav Kj\ae rbye Bagger}
\address{School of Science, UNSW Canberra, Australia}
\email{g.bagger@unsw.edu.au}
\date\today
\keywords{}

\begin{abstract}
Let $x$ and $n$ be positive integers. We prove a non-trivial lower bound for $x$, dependant only on $\omega_n$, the number of distinct prime factors of $x^n-1$. By considering the divisibility of $\varphi \mid x^n-1$ for $\varphi \mid n$, we obtain a further refinement. This bound has applications for existence problems relating to primitive elements in finite fields.
\end{abstract}

\maketitle
\blfootnote{\textit{Affiliation}: School of Science, The University of New South Wales Canberra, Australia.}
\blfootnote{\textit{Corresponding author}: Gustav Kj\ae rbye Bagger (g.bagger@unsw.edu.au).}
\blfootnote{\textit{Key phrases}: Congruences, Primes in arithmetic progressions, Cyclotomy}

\section{Introduction}
Consider the following problem. Given only the number $\omega_n$ of distinct prime factors of $x^n-1$, what lower bound can one give for $x$? One might bound $x$ via 
\begin{equation}
x^n-1 \geq \prod_{i=1}^{\omega_n} s_i \quad \text{ implying } \quad x \geq \left( 1 + \prod_{i=1}^{\omega_n} s_i \right)^{1/n}, \label{trivialb}
\end{equation}
the product over the $\omega_n$ smallest primes $s_i$. We shall refer to this bound as the \textit{trivial bound} on $x$. In 2022, Booker, Cohen, Leong and Trudgian \cite{BCLT} considered the problem of finding elements with prescribed traces in certain $\fin{}$ extensions. In the special case where $n \in \{3,5\}$ and $q$ a prime power, they proved a hybridised bound for $q$ in terms of certain residue classes. This paper generalises their approach and extends the hybrid bound to any positive integers $n$ and $x$. In particular, we prove the following bound.
\begin{Th} (\textit{Hybrid bound}) \label{hybridb}
Let $n=\prod_{i=1}^m \varphi_i^{a_i}$, $x \in \bb{Z}_{>1}$ and $k_i \in \bb{Z}^+$. Then
$$ x \geq \left. \left. \text{min}\left[\text{max}\left\{\left(1+ \prod_{j=0}^t\prod_{\substack{ i = 1 \\ r_{i,j} \in \Pi^j_{n} }}^{k_j} r_{i,j}\right)^{T(t)} \ \right\vert \ t \in [0,m]\right\} \ \right\vert \ \sum_{s=0}^m k_s= \omega_n \right],  $$
where $\Pi_n^j = \{ \rho \in \bb{P} \mid \rho \equiv 1 \mod{\varphi_j} \wedge \rho \not\equiv 1 \mod{\varphi_m,...,\varphi_{j+1}}\}$ for $j \in \{0,...,m\}$, $\varphi_0 := 1$, $T(t)=\prod_{k=1}^{t}\varphi_k^{-a_k}$ and where $r_{i,j}$ denotes the $i^\text{th}$ element of $\Pi_n^j$.
\end{Th}
To emphasise the improvement Theorem \ref{hybridb} provides over the trivial bound \eqref{trivialb}, consider the case where $n=3$ and $\omega_3 =10$. Here the trivial bound gives $x^3-1 \geq \prod_{i=1}^{10}s_i $, implying $x \geq 1864$. In contrast, Theorem \ref{hybridb} takes into account the factorisation $x^3-1 = (x-1)(x^2 + x + 1)$ in $\bb{Z}[x]$ and the fact that any prime divisor $\rho \mid x^2+x+1$ satisfies $\rho \equiv 0,1 \mod{3}$. Thus, either $\rho$ is in a restricted congruence class or $\rho \mid x-1$ and we call Theorem \ref{hybridb} a \textit{hybrid bound} since it is a hybrid between bounding $x-1$ and $x^n-1$ simultaneously. Playing these criteria on $\rho$ off against each other results in the bound $x \geq 3032$, a $63\%$ improvement over \eqref{trivialb}. Based on calculations, for fixed $n \neq 2^a$, the hybrid bound is expected to outperform the trivial bound at an ever increasing rate as $\omega_n \to \infty$. For fixed $\omega_n$ and varying $n$, calculations also support the assertion that the largest relative improvement is obtained when $n$ is prime. In 2010, Cohen \cite{Cohen} considered the problem of finding primitive elements on certain lines in $\fin{n}$ for small $n$. In the special case where $n=4$ and $q$ is a prime power, Cohen noted that $q$ odd implies $2^4 \mid q^4-1$. It follows that one obtains an improvement over \eqref{trivialb}:
$$q \text{ odd implies } \quad q \geq \left(1+ 8 \prod_{i=1}^{\omega_n} s_i\right)^{1/4}. $$
This result is suggestive of additional structure arising from divisibility $\varphi \mid x^n-1$ for $\varphi \mid n$. This paper generalises Cohen's approach to any positive integers $n$ and $x$, obtaining a refinement over the hybrid bound.
\begin{Th} (\textit{Ramification bound}) \label{ramifb}
Let $n=\prod_{i=1}^m \varphi_i^{a_i}$, $x \in \bb{Z}_{>1}$ and $k_i \in \bb{Z}^+$. For any prime factor $\varphi_\alpha$ of $n$, we have
$$ x \geq \left. \left. \text{min}\left[\text{max}\left\{\left(1+ \prod_{j=0}^t\delta_\alpha(j)\prod_{\substack{ i = 1 \\ r_{i,j} \in \Pi^j_{n} }}^{k_j} r_{i,j}\right)^{T(t)} \ \right\vert \ t \in [0,m]\right\} \ \right\vert \ \sum_{s=0}^m k_s= \omega_n \right], $$
where $T(t)=\prod_{k=1}^{t}\varphi_k^{-a_k}$ and $\delta_\alpha(j)$ is defined as
 \begin{align*}
\delta_\alpha(j) = & \left\{ 
\begin{array}{ccc}
\varphi_\alpha/r_{k_\tau,\tau} & \vline & j = \tau \wedge k_\tau < b \\
\varphi_\alpha^{a_\alpha} & \vline & j = \alpha \\
1 & \vline & \text{elsewhere}
\end{array} 
\right. & \text{ for } \varphi_\alpha \mid x^n-1 \\
\delta_\alpha(j) = & \left\{ 
\begin{array}{ccc}
r_{k_\tau+1,\tau}/\varphi_\alpha & \vline & j = \tau \wedge k_\tau \geq b \\
1 & \vline & \text{elsewhere}
\end{array} 
\right. & \text{ for } \varphi_\alpha \nmid x^n-1.
\end{align*} 
Here, we take $\tau$ such that $\varphi_\alpha \in \Pi_n^\tau$ and $b$ such that $\varphi_\alpha = r_{b,\tau}$. 
\end{Th}
Returning to the example where $n=3$ and $\omega_3=10$, Theorem \ref{ramifb} gives:
$$ x \geq B = \left\{ \begin{array}{ccc}
4387 & \text{ for } & 3 \mid x^3-1  \\
5423 & \text{ for } & 3 \nmid x^3-1 
\end{array} \right. .$$
In both divisibility cases, we obtain a significant improvement over both the hybrid and trivial bounds. In the special case where $n=2^a$ for some $a \in \bb{Z}^+$, the hybrid bound coincides with the trivial bound. In this setting, we instead prove a modified version of the ramification bound.
\begin{Th}
Let $x,a \in \bb{Z}^+$ and $n = 2^a$. Then \label{modramb}
$$ x \geq \text{min}\left[\left(1+ 2 n \cdot \prod\limits_{i=1}^{\omega_n} s_i \right)^{1/n},\left(1+ \prod\limits_{i=2}^{\omega_n+1} s_i \right)^{1/n}\right]. $$ 
\end{Th}
If we consider $n=8$ and $\omega_8=20$, then \eqref{trivialb} gives the bound  $x \geq 2205$. In contrast, Theorem \ref{modramb} results in the bound $x \geq 3118$, a $41 \%$ improvement. For fixed $n=2^a$ and $\omega_n$ sufficiently large, Theorem \ref{modramb} outperforms the trivial bound by a factor of $(2n)^{1/n}$. We call Theorem \ref{ramifb} a \textit{ramification bound} due to the following observation: A prime $\rho \in \bb{Z}$ is ramified in the cyclotomic extension $\bb{Q}(e^{2 \pi i / n})/ \bb{Q}$ if and only if $\rho \mid n$. See Washington \cite[p. 10]{washington} for a proof of this fact. Bounds dependant on $\omega_n$ have applications to numerous existence results concerning primitive elements in $\fin{}$ extensions. By applying our hybridised bounds, we can sometimes obtain immediate improvements to previous known results. We discuss two such applications in detail: Cohen's line problem \cite{Cohen} and Liao, Li and Pu's primitive elements sum problem \cite{LLP}. In general, sieves reliant on $\omega_n$ are natural settings for these hybridised bounds. See for example Lemos, Neumann and Ribas \cite{Lemos}, Rani et al. \cite{Rani}, Sharma \cite{Sharma2} and \cite{Sharma1} or Takshak et al. \cite{Takshak}. \\

The outline of this paper is as follows. In \sect 2, we develop some elementary lemmas. In \sect 3, we utilise the lemmas to prove the hybrid and ramification bounds in the setting where $n$ is prime. In \sect 4, we extend the lemmas to composite $n$ and prove a useful factorisation of $x^n-1$ in $\bb{Z}[x]$. The proof of Theorem \ref{hybridb} then follows from these observations. In \sect 5 we prove Theorem \ref{ramifb} and Theorem \ref{modramb}. In \sect 6, we conclude with applications to finite field theory.

\section{Preliminaries}
Let $\bb{P}$ denote the set of rational primes. Throughout this section, we assume that $n = \varphi$ for some $\varphi \in \bb{P}$. This will provide the base case for the general construction. Denote the $k^\text{th}$ cyclotomic polynomial by $\Phi_k(x)$ and recall that $ \Phi_\varphi(x) = \sum_{i=0}^{\varphi-1}x^{i}$ in $\bb{Z}[x]$. We have the $\bb{Z}[x]$-factorisation
$$ x^\varphi-1 = \prod_{d \mid \varphi} \Phi_d(x) = \Phi_1(x) \Phi_\varphi(x) = (x-1) \Phi_\varphi(x). $$
Our aim will be to classify the prime divisors of $x^\varphi-1$ in terms of the divisors of $x-1$ and $\Phi_\varphi(x)$. To that end, we have the following lemma.
\begin{Lemma}
Let $\rho,\varphi \in \bb{P}$ and $x\in \bb{Z}^+$. If \label{onlyp}
$\rho \mid x-1,\Phi_\varphi(x) $ then $ \rho=\varphi$.
\end{Lemma}
\begin{proof}
If $\rho \mid x-1$ then $x \equiv 1 \mod{\rho}$. Thus $ 0 \equiv \Phi_\varphi(x) \equiv \Phi_\varphi(1) \equiv \varphi \mod{\rho}$. Then $\rho \mid \varphi $, implying $\rho=\varphi$, as required. \qedhere
\end{proof}
Thus, apart from potential twice-divisibility by $\rho$, no other prime factor divides both $x-1$ and $\Phi_\varphi(x)$. In fact, if $\varphi$ divides $x^\varphi-1$, then $\varphi^2$ divides $x^\varphi-1$. This follows from the lemma below.
\begin{Lemma}
Let $\varphi \in \bb{P}$ and $x\in \bb{Z}^+$. Then \label{bothornone}
$\varphi \mid x^\varphi-1$ if and only if $\varphi \mid x-1,\Phi_\varphi(x)$.
\end{Lemma}
\begin{proof} 
We have $ \varphi \mid x^\varphi-1 $ if and only if $ x^\varphi-1 \equiv 0 \mod{\varphi}$. By Fermat's little theorem, we have $ 0 \equiv x^\varphi-x \equiv 1-x \mod{\varphi}, $ implying $\varphi \mid x-1$. Therefore, we have $ \Phi_\varphi(x) \equiv \Phi_\varphi(1) \equiv \varphi \equiv 0 \mod{\varphi} $, as required. The reverse implication follows trivially from the factorisation of $x^\varphi-1$ in $\bb{Z}[x]$. \qedhere

\end{proof}
A priori, it is not clear why it should be advantageous to factorise $x^\varphi-1$ in $\bb{Z}[x]$ in this way. In order to leverage this form, we require a necessary condition for divisibility of $\Phi_\varphi(x)$. 
\begin{Prop}
Let $\rho,\varphi \in \bb{P}$ with $\rho \neq \varphi$ and $x\in \bb{Z}^+$. Then \label{arithprog}
$\rho \mid \Phi_\varphi(x)$ implies $\rho \equiv 1 \mod{\varphi}$.
\end{Prop}
\begin{proof}
Assume $\rho \mid \Phi_\varphi(x)$. By Lemma \ref{onlyp}, $\varphi$ cannot divide $x-1$ since this would imply $\rho = \varphi$, a contradiction. Indeed $ x^\varphi -1 = (x-1) \Phi_\varphi(x) \equiv 0 \mod{\rho}$ implies $x^\varphi \equiv 1 \mod{\rho}$.
Then, the multiplicative order of $x$ in $(\bb{Z}/\rho\bb{Z})^\times$ divides $\varphi$, implying $\text{ord}(x) \in\{1,\varphi\}$. If $\text{ord}(x) = 1$, then $\rho \mid x-1$, a contradiction. Thus $\text{ord}(x)=\varphi$ and, by Lagrange's theorem, $\varphi \mid \rho-1$, proving the claim. \qedhere

\end{proof}
Note that, by the contrapositive statement, if $\rho \mid x^\varphi-1$, for some prime different from $\varphi$, then we obtain a sufficient condition for $x-1$ divisibility. That is, if $\rho \not\equiv 1 \mod{\varphi}$, then we must have $\rho \mid x-1$. 

\section{The prime case}
We are now in a position to extend the hybrid bound to any prime $n=\varphi$. The effectiveness of the bound fundamentally relies on the following observation.
\begin{Rem} \label{hybidea} 
For any prime divisor $\rho \mid x^\varphi - 1$ outside of $\varphi$ itself, $\rho$ divides either $x-1$ or $\Phi_\varphi(x)$. If $\rho$ divides $\Phi_\varphi(x)$ then $\rho$ is in a sparse subset of the primes, namely in the arithmetic progression $ \rho \equiv 1 \mod{\varphi}$. If most of the prime divisors are of this form, the bound on $x$ is improved since each factor in the product bounding $x^\varphi-1$ is larger. If instead many prime divisors of $x^\varphi-1$ also divide $x-1$, then we can bound $x-1$ by a product of small primes instead of $x^\varphi-1$. This, in turn, improves the bound on $x$. Consider the case where $n=3$ and $\omega_3=6$. The bound \eqref{trivialb} gives $x \geq 32$. Let $k \in \{0,\dots,6\}$ denote the number of primes $\rho \mid x^3-1$ where $\rho \equiv 2 \mod{3}$. If we assume $k=2$, then we can bound $x$ in two ways: 
\begin{align*}
x^3 - 1 \geq & \ (2 \cdot 5 )\cdot (3 \cdot 7 \cdot 13 \cdot 19 ) \\
x - 1 \geq & \ (2 \cdot 5).
\end{align*}
This gives the bound $x\geq (51871)^{1/3}>37$. If instead we take $k=3$, we obtain
\begin{align*}
x^3 - 1 \geq & \ (2 \cdot 5 \cdot 11) \cdot (3 \cdot 7 \cdot 13 ) \\
x - 1 \geq & \ (2 \cdot 5 \cdot 11).
\end{align*} 
This gives the bound $x\geq 110$. If we continue in the same manner over all values of $k$, we find that the bound is minimal when $k=2$. Since, $k$ is unknown for general $x$, we minimise over $k$ and conclude that $x \geq 38$. 
\end{Rem}
In order to formalise Remark \ref{hybidea}, we partition the primes via $\varphi$. Define $ \Pi_\varphi^0 := \{ \rho \in \bb{P} \mid \rho \not\equiv 1 \mod{\varphi} \} $ and $ \Pi_\varphi^1 := \{ \rho \in \bb{P} \mid \rho \equiv 1 \mod{\varphi} \}$. By construction, we must have $ \bb{P} = \Pi_\varphi^0 \sqcup \Pi_\varphi^1 $. Moreover, define the sets $\Omega_\varphi(x) := \{ \rho \in \bb{P} \text{ st } \rho \mid x^\varphi-1 \}$ and $\Omega_\varphi^i(x) :=  \Omega_\varphi(x) \cap \Pi_\varphi^i $ for $i=0 $ or $1$. Under the canonical ordering of the primes, denote the $i^\text{th}$ element of $\Pi_\varphi^0$ and $\Pi_\varphi^1$ by $r_{i,0}$ and $r_{i,1}$, respectively. Using this notation, we have $\omega_\varphi = \abs{\Omega_\varphi(x)}$. Then, we can split the bound for $x^\varphi-1$ into disjoint factors according to $\Omega_\varphi^i(x)$ membership:
$$ x^\varphi -1  \geq \prod_{\rho \in \Omega_\varphi(x)} \rho = \prod_{\rho_0 \in \Omega_\varphi^0(x)} \rho_0 \prod_{\rho_1 \in \Omega_\varphi^1(x)} \rho_1. $$
By Proposition \ref{arithprog}, we have $\rho_0 \mid x-1$ for any $\rho_0 \in \Omega_\varphi^0(x)$. Thus, $x$ must also satisfy $ x - 1 \geq \prod_{\rho_0 \in \Omega_\varphi^0(x)} \rho_0$. Combining these bounds, we obtain
\begin{equation}
x \geq \text{max}\left\{ 1+ \prod_{\rho_0 \in \Omega_\varphi^0(x)} \rho_0,\left( 1 + \prod_{\rho_0 \in \Omega_\varphi^0(x)} \rho_0 \prod_{\rho_1 \in \Omega_\varphi^1(x)} \rho_1 \right)^{1/\varphi} \right\}. \label{boundprime}
\end{equation}  
By considering pointwise bounds on these products, we obtain a general bound. 
\begin{Prop} (\textit{$\varphi$-hybrid bound}) Let $\varphi \in \bb{P}$ and $x \in \bb{Z}_{>1}$. Then \label{phybbound}
$$ x \geq \text{min}_{0 \leq k < \omega_\varphi}\left[ \text{max}\left\{ 
1+\prod\limits_{\substack{i=1 \\ r_{i,0} \in \Pi_\varphi^0}}^k r_{i,0}, \left(1 + \prod\limits_{\substack{i=1 \\ r_{i,0} \in \Pi_\varphi^0}}^k r_{i,0} \prod\limits_{\substack{j=1 \\ r_{j,1} \in \Pi_\varphi^1}}^{\omega_\varphi-k} r_{j,1}\right)^{1/\varphi} \right\} \right]. $$
\end{Prop}
\begin{proof}
Assume that $k = \abs{\Omega_\varphi^0}$, thus $\abs{\Omega_\varphi^1} = \omega_\varphi-k$. By choosing the smallest possible candidates for $\rho_0 \in \Omega_\varphi^0(x)$ and $\rho_1 \in \Omega_\varphi^1(x)$, we obtain the inequalities $ \prod_{\rho_0 \in \Omega_\varphi^0(x)} \rho_0 \geq \prod_{i=1}^k r_{i,0} $ and $ \prod_{\rho_1 \in \Omega_\varphi^1(x)} \rho_1 \geq \prod_{j=1 }^{\omega_\varphi-k} r_{j,1}$. We apply the pointwise bounds for $\prod_{\Omega_\varphi^0(x)}$ and $\prod_{\Omega_\varphi^1(x)}$ in \eqref{boundprime} at each $k$ via the argument above. Thus, we are left with 
$$ x \geq  \text{max}\left\{ 
1+\prod\limits_{\substack{i=1 \\ r_{i,0} \in \Pi_\varphi^0}}^k r_{i,0}, \left(1 + \prod\limits_{\substack{i=1 \\ r_{i,0} \in \Pi_\varphi^0}}^k r_{i,0} \prod\limits_{\substack{j=1 \\ r_{j,1} \in \Pi_\varphi^1}}^{\omega_\varphi-k} r_{j,1}\right)^{1/\varphi} \right\}. $$
Note that $ \abs{\Omega_\varphi^1}>0$ for any $x$, so $k \neq \omega_\varphi$. Since $k \in [0,\omega_\varphi)$ is unknown, we minimise over $k$ and arrive at the statement. \qedhere
\end{proof}

We are able to improve this further by considering the divisibility of $x^\varphi-1$ by $\varphi$. This follows from Lemmas \ref{onlyp} and \ref{bothornone}. If $\varphi \mid x^\varphi-1$, then $\varphi^2 \mid x^\varphi-1$. Moreover, if $\varphi \nmid x^\varphi-1$, we exclude $\varphi$ from our lower bound, replacing it with a larger prime factor. For ease of notation, define $R^0_k := \prod_{i=1}^k r_{i,0} $ and $ R^1_k := \prod_{j=1 }^{\omega_\varphi-k} r_{j,1}$, with the convention where $r_{i,s} \in \Pi_\varphi^s$ denotes the $i^\text{th}$ element under the natural ordering.

\begin{Prop} (\textit{$\varphi$-ramification bound}) \label{pramb}
Let $\varphi \in \bb{P}$ with $x \in \bb{Z}_{>1}$ and assume, without loss of generality, that $\varphi = r_{k'} \in \Pi_\varphi^0$. Let $I_{k'}^- = [0,k')$ and $ I_{k'}^+= [k',\omega_\varphi)$. Then we have
\begin{align*}
x \geq & \ \text{min}\left[ \left\{ \text{max}\left\{ 
1+R^0_k, \left(1 + R^0_k R^1_k\right)^{1/\varphi} \right\} \ \biggr\vert \ k \in I_{k'}^- \right\}, \right. & \\
 & \left. \left\{ \text{max}\left\{ 
1+\varphi^{-1}R^0_{k+1}, \left(1 + \varphi^{-1}R^0_{k+1}R^1_k\right)^{1/\varphi} \right\} \ \biggr\vert \ k \in I_{k'}^+ \right\} \right] & \ \ \text{for }  \varphi \notin \Omega_p(x) \ \\
x \geq & \ \text{min}\left[ \left\{ \text{max}\left\{ 
1+\varphi R^0_{k-1}, \left(1 + \varphi^2 R^0_{k-1} R^1_k \right)^{1/\varphi} \right\} \ \biggr\vert \ k \in I_{k'}^- \right\}, \right. & \\
&  \left. \left\{ \text{max}\left\{ 
1+R^0_k, \left(1 + \varphi R^0_k R^1_k\right)^{1/\varphi} \right\}  \ \biggr\vert \ k \in I_{k'}^+ \right\} \right] & \ \ \text{for } \varphi \in \Omega_p(x).  
\end{align*} 
\end{Prop}
\begin{proof}
Consider the general bound obtained in \eqref{boundprime}. We subdivide the proof into cases dependant on $\varphi$-membership of $\Omega_\varphi(x)$. \\ \ \\
\underline{Case 1}: $\varphi \notin \Omega_\varphi(x)$. \\
\bf{1a) } If $k < k'$, then, since $\varphi$ is not a factor of $R^0_k$, we use the hybrid bound 
$$ x \geq \text{max}\left\{ 
1+R^0_k, \left(1 + R^0_k R^1_k\right)^{1/\varphi} \right\}. $$
\bf{1b) } If instead $k' \leq k<\omega_\varphi$, then $\varphi$ appears as a factor of $R^0_k$. Therefore, we can exclude $\varphi$ from the product and instead consider $\varphi^{-1}R^0_{k+1}$, resulting in the bound
$$ x \geq \text{max}\left\{ 
1+\varphi^{-1}R^0_{k+1}, \left(1 + \varphi^{-1}R^0_{k+1}R^1_k\right)^{1/\varphi} \right\}. $$
We minimise over possible values of $k$, obtaining the desired bound. \\ \ \\
\underline{Case 2}: $\varphi \in \Omega_\varphi(x)$. \\
\bf{2a) } If $k < k'$, then all factors of $R^0_k$ will be smaller than $\varphi$, thus we can replace the largest such factor $r_{k,0}$ by $\varphi$. Moreover, by Lemma \ref{bothornone}, if $\varphi \in \Omega_\varphi(x)$, then $\varphi^2 \mid x^\varphi-1$. Thus, we can replace $R^1_k$ by $\varphi R^1_k$. Then
$$ x \geq \text{max}\left\{ 
1+\varphi R^0_{k-1}, \left(1 + \varphi^2 R^0_{k-1} R^1_k \right)^{1/\varphi} \right\}. $$
\bf{2b) } If instead $k' \leq k < \omega_\varphi$, then $\varphi$ already appears as a factor in $R^0_k$ but we still gain an extra $\varphi$ factor in $R^0_k R^1_k$ by the argument in case \bf{2a}. Thus 
$$ x \geq \text{max}\left\{ 
1+R^0_k, \left(1 + \varphi R^0_k R^1_k\right)^{1/\varphi} \right\}. $$
We minimise over possible values of $k$, obtaining the desired bound. \qedhere
\end{proof}
Returning to the example given in Remark \ref{hybidea}, let $k=2$ and $3 \nmid x^3-1$. Then
\begin{align*}
x^3 - 1 \geq & \ (2 \cdot 5 )\cdot ( 7 \cdot 13 \cdot 19 \cdot 31) \\
x - 1 \geq & \ (2 \cdot 5).
\end{align*}
This gives the bound $x\geq (535990)^{1/3}>81$. If instead $3 \mid x^3-1$, we obtain
\begin{align*}
x^3 - 1 \geq & \ 3 \cdot (2 \cdot 5 ) \cdot (3 \cdot 7 \cdot 13 \cdot 19) \\
x - 1 \geq & \ 3 \cdot (2 \cdot 5),
\end{align*}
resulting in the bound $x (155610)^{1/3}>53$. Both divisibility cases provide a significant improvement over the bound in Remark \ref{hybidea}.

\section{The generalised hybrid bound}
We generalise results from the previous sections to the setting where we consider $x^n-1$ with $n$ composite. The first step is to factorise $x^n-1$ in $\bb{Z}[x]$ such that we obtain a natural partition of the prime divisors. 
\begin{Prop}
Let $n = \prod_{i=1}^m \varphi_i^{a_i}$ and define $x_k=\exp\{\log(x) \prod_{r=1}^{k-1} \varphi_r^{a_r}\}$ for any choice of $k \in [1,m]$. Then we have the $\bb{Z}[x]$-factorisation \label{sensiblefact}
$$ x^n-1 = (x-1) \prod_{i=1}^m \left[ \prod_{j=1}^{a_i} \Phi_{\varphi_i}\left( x_i^{\varphi_i^{j-1}} \right)\right]. $$
\end{Prop}
\begin{proof}
Consider the factorisation $n = r \cdot \varphi_m^{a_m}$. Then, we have
 \begin{align*}
x^n-1 = & \ (x^{n/\varphi_m})^{\varphi_m} -1 \\
 = & \ (x^{n/\varphi_m}-1) \Phi_{\varphi_m}\left(x^{r\cdot \varphi_m^{m-1}} \right) \\
 = & \ (x^{n/\varphi^2_m}-1)  \Phi_{\varphi_m}\left(x^{r\cdot \varphi_m^{m-2}}\right) \cdot \Phi_{\varphi_m}\left(x^{r\cdot \varphi_m^{m-1}}\right) \\
  \vdots &  \\
 = & \ (x^r - 1) \Phi_{\varphi_m}(x^r) \cdot ...\cdot \Phi_{\varphi_m}\left(x^{r\cdot \varphi_m^{m-1}}\right) \\
 = & \ (x^r-1) \prod\limits_{j=1}^{a_m}\Phi_{\varphi_m}\left(x^{r\cdot \varphi_m^{j-1}}\right).
\end{align*}
Iterating the factorisation on $x_k-1$ for $k \in \{m,\dots,1\}$, we have
\begin{align*}
x^n-1  = & \ (x_m-1) \prod\limits_{j=1}^{a_m}\Phi_{\varphi_m}\left(x^{r\cdot \varphi_m^{j-1}}\right) \\
= & \ (x_{m-1}-1) \left[ \prod\limits_{j=1}^{a_{m-1}}\Phi_{\varphi_{m-1}}\left(x_{m-1}^{ \varphi_{m-1}^{j-1}}\right) \right] \left[\prod\limits_{j=1}^{a_m}\Phi_{\varphi_{m}}\left(x_{m}^{ \varphi_{m}^{j-1}}\right) \right] \\
= & \ (x-1) \left[ \prod\limits_{j=1}^{a_{1}}\Phi_{\varphi_{1}}\left(x_1^{ \varphi_{1}^{j-1}}\right) \right] \cdot ... \cdot \left[\prod\limits_{j=1}^{a_m}\Phi_{\varphi_{m}}\left(x_{m}^{ \varphi_{m}^{j-1}}\right) \right] \\
= & \ (x-1) \prod_{i=1}^m\left[\prod\limits_{j=1}^{a_i}\Phi_{\varphi_{i}}\left(x_{i}^{ \varphi_{i}^{j-1}}\right) \right] .
\end{align*} \qedhere
\end{proof}
Note that, if we define $x_{m+1}:=x^n$, for any $s \in [1,m]$ we have the relation
$$ x_{s+1} -1 = (x_{s}-1) \prod_{j=1}^{a_s}\Phi_{\varphi_s}\left( x_s^{{\varphi_s}^{j-1}}\right), $$
in the polynomial ring $\bb{Z}[x]$. We proceed by generalising Lemmas \ref{onlyp} and \ref{bothornone}. 
\begin{Lemma}
Let $\rho,\varphi \in \bb{P}$ and $y \in \bb{Z}^+$. Then $ \rho \mid y-1,\Phi_{\varphi}\left(y^{\varphi^{a-1}}\right) $ for some $a \in \bb{Z}^+$ implies we must have $\rho = \varphi$. \label{onlypgen}
\end{Lemma}
\begin{proof}
Assume $ \rho \mid y-1,\Phi_{\varphi}\left(y^{ \varphi^{a-1}}\right) $. Then, $y \equiv 1 \mod{\rho}$, which implies that 
$$ \Phi_\varphi\left(y^{\varphi^{a-1}} \right) \equiv \Phi_\varphi\left(1^{\varphi^{a-1}} \right) \equiv \Phi_\varphi\left(1\right) \equiv \varphi \mod{\rho}. $$
Therefore $ \rho \mid \Phi_{\varphi}\left(y^{ \varphi^{a-1}}\right) $ implies that $\rho \mid \varphi$. Since both $\rho$ and $\varphi$ are primes, we conclude that indeed $\rho = \varphi$. \qedhere
\end{proof}
\begin{Rem}
Consider the situation where $n = r \varphi_1^{a_1}\varphi_2^{a_2}$ for two distinct $\varphi_1,\varphi_2 \in \bb{P}$ where $\varphi_1,\varphi_2 \nmid r$. We have the $\bb{Z}[x]$-factorisation
$$ x^n-1 = (x^r-1) \prod_{i=1}^{a_1} \Phi_{\varphi_1}\left((x^r)^{\varphi_1^{i-1}} \right) \prod_{j=1}^{a_2} \Phi_{\varphi_2}\left((x^{r\varphi_1^{a_1}})^{\varphi_2^{j-1}} \right). $$
Assume there exists some $k_1,k_2$ and some prime $\rho$ such that
$$ \rho \mid \Phi_{\varphi_1} \left((x^r)^{\varphi_1^{k_1-1}} \right),\Phi_{\varphi_2}\left((x^{r\varphi_1^{a_1}})^{\varphi_2^{k_2-1}} \right). $$
If we set $y = x^{r\varphi_1^{a_1}} $, then we have $ \left. \Phi_{\varphi_1} \left((x^r)^{\varphi_1^{k_1-1}} \right) \ \right\vert \ y - 1 $. Thus we obtain $ \rho \mid y-1,\Phi_{\varphi_2}\left(y^{\varphi_2^{k_2-1}}\right) $ which, by Lemma \ref{onlypgen}, implies $ \rho = \varphi_2$. This shows that the only prime divisors potentially dividing cyclotomic polynomials with different bases in the $\bb{Z}[x]$-factorisation from Proposition \ref{sensiblefact} are the bases themselves.
\end{Rem}
\begin{Lemma}
Let $\varphi \in \bb{P}$ and $n = r \cdot \varphi^a $ with $(r,\varphi) = 1$. Then
$$ \varphi \mid x^n - 1 \quad \text{ if and only if } \quad \varphi \mid x^r-1,\Phi_{\varphi}\left(x^{r \cdot \varphi^{i-1}}\right) 
 \quad , \ i \in \{1,...,a\}.$$   

\end{Lemma}
\begin{proof}
Note that, by definition, we have
$$ \varphi \mid x^n-1 \quad \text{ if and only if } \quad x^n-1 = (x^r)^{\varphi^a} - 1\equiv 0 \mod{\varphi}. $$
By Fermat's little theorem, we observe that $ 0 \equiv (x^r)^\varphi-(x^r) \equiv (x^r)^{\varphi^a} - x^r \mod{\varphi}$, implying $x^r \equiv (x^r)^{\varphi^a} \equiv 1 \mod{\varphi}$. We conclude that $\varphi \mid x^r-1$. Moreover, if $ x^r \equiv 1 \mod{\varphi} $, then, for any fixed $i \in \{1,...,a\}$, we must have $\Phi_{\varphi}\left(x^{r \cdot \varphi^{i-1}}\right) \equiv \Phi_p\left(1^{\varphi^{i-1}}\right) \equiv \varphi \mod{\varphi}$. Since $\varphi \mid \varphi$, the $\varphi$-divisibility of $\Phi_{\varphi}\left(x^{r \cdot \varphi^{i-1}}\right)$ is clear. The reverse implication follows directly from Proposition \ref{sensiblefact}. \qedhere
\end{proof}
Reminiscent of the $n \in \bb{P}$ case, we partition the set of primes. For a given $n = \prod\limits_{i=1}^m \varphi_i^{a_i}$ and any $j \in [1,m-1]$, define the sets
\begin{flalign*}
\Pi_n^m := & \{ \rho \in \bb{P} \mid \rho \equiv 1 \mod{\varphi_m} \} \\
\Pi_n^j := & \{ \rho \in \bb{P} \mid \rho \equiv 1 \mod{\varphi_j} \wedge \rho \not\equiv 1 \mod{\varphi_m,...,\varphi_{j+1}}\} \\ 
\Pi_n^0 := & \{ \rho \in \bb{P} \mid \rho \not\equiv 1 \mod{\varphi_m,...,\varphi_1}\}.
\end{flalign*}
By construction, $ \bb{P} = \bigsqcup_{j=0}^m \Pi_n^j $. Moreover, for $j \in [0,m]$, define the sets
$ \Omega_n(x) := \{ \rho \in \bb{P} \text{ st } \rho \mid x^n-1 \} $ and $\Omega_n^j(x) := \Omega_n(x) \cap \Pi_n^j$. Under the canonical ordering of the primes, denote the $i^\text{th}$ element of $\Pi_n^j$ by $r_{i,j}$. Using this notation, we have $\omega_n = \abs{\Omega_n(x)}$. Then we have the inequality
$$ x^n -1  \geq \prod_{\rho \in \Omega_{n}(x)} \rho = \prod_{\rho_0 \in \Omega_n^0(x)} \rho_0 \cdots \prod_{\rho_m \in \Omega_n^m(x)} \rho_m. $$
Recall the factorisation $ x^n-1 = (x_m-1) \prod_{j=1}^{a_m}\Phi_{\varphi_m}\left(x_m^{j-1}\right)$ in $\bb{Z}[x]$. If $\rho \mid x^n-1$ and $\rho \not\equiv 1 \mod{p_m}$, we must have $\rho \mid x_m-1$. Thus $\rho \in \Omega_n^j(x)$ for some $j \in [0,m-1]$. Iterating this process, we obtain the collection of bounds
\begin{align*}
x^n-1 \geq & \ \prod\limits_{\rho_0 \in \Omega_n^0(x)} \rho_0 \cdots \prod\limits_{\rho_m \in \Omega_n^m(x)} \rho_m  \\
x_m - 1 \geq & \ \prod\limits_{\rho_0 \in \Omega_n^0(x)} \rho_0 \cdots \prod\limits_{\rho_{m-1} \in \Omega_n^{m-1}(x)} \rho_{m-1}  \\
\vdots &  \\
x-1 \geq & \ \prod\limits_{\rho_0 \in \Omega_n^0(x)} \rho_0.
\end{align*} 
Combining these, we have the general bound
\begin{equation}
 x \geq \text{max}\left\{ 1 +\prod\limits_{\rho_0 \in \Omega_n^0(x)} \rho_0, ...,\left(1 + \prod\limits_{\rho_0 \in \Omega_n^0(x)} \rho_0 \cdot ... \cdot \prod\limits_{\rho_m \in \Omega_n^m(x)} \rho_m \right)^{1/n} \right\}. \label{genhyb}\end{equation}
We are now in a position to prove Theorem \ref{hybridb} as a generalisation of Proposition \ref{phybbound} for composite values of $n$.
\begin{proof}[Proof of Theorem \ref{hybridb}]
Set $k_s = \abs{\Omega_n^s(x)} \in \bb{Z}^+$ for all $s \in \{0,...,m\}$. By construction, we have $\sum_{s=0}^m k_s= \omega_n$. Moreover, by choosing the smallest possible candidates for $r_j \in \Omega_{n}^j(x)$, we obtain the inequalities $\prod_{r_j \in \Omega_n^j(x)} r_j \geq \prod_{i = 1}^{k_j} r_{i,j} $ for all $j \in \{0,...,m\}$. Applying the pointwise inequalities to the bound \eqref{genhyb}, we obtain
$$ x \geq \left. \text{max}\left\{\left(1+ \prod_{j=0}^t\prod_{\substack{ i = 1 \\ r_{i,j} \in \Pi^j_{n} }}^{k_j} r_{i,j}\right)^{T(t)} \ \right\vert \ t \in [0,m] \right\}.  $$
A priori, the distribution of $k_s$ is arbitrary, so we minimise over all valid values of the tuple $(k_0,...,k_m) \in \bb{Z}^m$, obtaining the desired result. \qedhere
\end{proof}
Note that this form of the hybrid bound is the correct generalisation in the sense that it coincides with the prime bound in the special case when $n=\varphi$. Indeed
\begin{align*} x \geq & \ \left. \left. \text{min}\left[\text{max}\left\{\left(1+ \prod\limits_{j=0}^t\prod\limits_{\substack{ i = 1 \\ r_{i,j} \in \Pi^j_{n} }}^{k_j} r_{i,j}\right)^{T(t)} \ \right\vert \ t \in [0,m]\right\} \ \right\vert \ \sum\limits_{s=0}^m k_s= \omega_n \right]  \\
 = & \ \left. \left. \text{min}\left[\text{max}\left\{\left(1+ \prod\limits_{j=0}^t\prod\limits_{\substack{ i = 1 \\ r_{i,j} \in \Pi^j_{\varphi} }}^{k_j} r_{i,j}\right)^{T(t)} \ \right\vert \ t \in [0,1]\right\} \ \right\vert \ k_0+k_1= \omega_n \right] \\
= & \ \ \ \text{min}_{k_0+k_1= \omega_n}\left[\text{max}\left\{1+ \prod\limits_{\substack{ i = 1 \\ r_{i,0} \in \Pi_\varphi^0 }}^{k_0} r_{i,0}, \left(1+ \prod\limits_{\substack{ i = 1 \\ r_{i,0} \in \Pi_\varphi^0 }}^{k_0} r_{i,0} \prod\limits_{\substack{ i = 1 \\ r_{i,1} \in \Pi_\varphi^1}}^{k_1} r_{i,1}\right)^{1/\varphi} \right\} \right] \\
= & \ \ \ \text{min}_{k_0+k_1= \omega_\varphi}\left[\text{max}\left\{1+ R^0_{k_0}, \left(1+ R^0_{k_0}R^1_{k_0} \right)^{1/\varphi} \right\} \right]. \end{align*}
This is a reformulation of the hybrid bound given in Proposition \ref{phybbound}.

\section{The generalised ramification bound}
In this section, we prove a generalisation of the ramification bound for composite $n$. This refinement will be dependant on the divisibility $\varphi \mid x^n-1$ for a given prime factor $\varphi \mid n$. As before, let $n = \sum_{i=1}^m \varphi_i^{a_i}$ where the factors $\varphi_i$ are ordered via $i< j $ implying $ \varphi_i < \varphi_j$ and $x \in \bb{Z}^+$. Choose a prime divisor $\varphi_\alpha \mid n$ and assume $\varphi_\alpha \in \Pi_n^\tau$. It is clear that $\tau < \alpha $ since $1 < \varphi_\alpha \leq \varphi_{r}$ implies $\varphi_\alpha \not\equiv 1 \mod{\varphi_r}$, which holds for any $r \geq \alpha$. 
\begin{Rem}
If $\varphi_\alpha \mid x^n-1$, then \label{palpha} $\varphi_\alpha \mid x_{\tau+1}-1 $  and $ \varphi_\alpha \mid \Phi_{\varphi_\alpha}\left(x_\alpha^{ \varphi_\alpha^{j-1}}\right) $ for any $ j \in [1,a_\alpha] $. Thus, $\varphi_\alpha$ must appear in the product $\prod\limits_{\rho_\tau \in \Omega_n^\tau(x)} \rho_\tau$. Moreover, when the product $\prod\limits_{\rho_\alpha \in \Omega_n^\alpha(x)} \rho_\alpha$ appears in our bound, we can replace this with $\varphi_\alpha^{a_\alpha}\prod\limits_{\rho_\alpha \in \Omega_n^\alpha(x)} \rho_\alpha$. The above argument gives the correct generalisation of the ramification bound in Proposition \ref{pramb}. We are thus in a position to prove Theorem \ref{ramifb}.
\end{Rem}
\begin{proof}[Proof of Theorem \ref{ramifb}]
Assume $\varphi_\alpha \mid x^n-1$. Then, by the argument in Remark \ref{palpha}, $\varphi_\alpha$ must appear in the product $\prod_{\rho_\tau \in \Omega_n^\tau(x)} \rho_\tau$. If $k_r \geq b$, then $\varphi_\alpha$ is already a factor of the product. If instead $k_r< b$, then we can replace the largest factor in the product by $\varphi_\alpha$. Moreover, $\prod_{\rho_\alpha \in \Omega_n^\alpha(x)} \rho_\alpha$ can be multiplied with $\varphi_\alpha^{a_\alpha}$ since $\varphi_\alpha \mid x^n-1$ implies $\varphi_\alpha \mid \Phi_{\varphi_\alpha}\left(x_\alpha^{ \varphi_\alpha^{j-1}}\right)$ for any $j \in [1,a_\alpha]$. If instead we assume $\varphi_\alpha \nmid x^n-1$, then we can exclude $\varphi_\alpha$ from the product $\prod_{\rho_\tau \in \Omega_n^\tau(x)} \rho_\tau$. Thus, whenever $k_\tau > b$, we can replace $\varphi_\alpha$ with the smallest prime $r_{k_\tau+1,\tau}$ in $\Pi_n^\tau$ which is not already a factor. In all other cases, we use the general hybrid bounds from Theorem \ref{hybridb}. \qedhere
\end{proof}
\begin{Rem} 
Whenever $n = 2^a$, Theorem \ref{hybridb} does not provide an improvement over \eqref{trivialb}. In essence the hybrid bound relies on maximising over bounds on $x^b-1$ for $b \mid n$. Unfortunately, we have $ \Pi_n^0 =  \{\rho \in \bb{P} \mid \rho \not\equiv 1 \mod{2} \} = \{2\} $. Thus, $k_0 \in \{0,1\}$ and we obtain
 \begin{align*} x \geq & \  \left. \text{min}\left[\text{max}\left\{1+\prod\limits_{i=1}^{k_0} 2,\left(1+ \prod\limits_{i=1}^{k_0} 2 \prod\limits_{\substack{ i = 1 \\ r_{i,1} \in \Pi^1_{n} }}^{\omega_n-k_0} r_{i,1}\right)^{1/n} \right\} \ \right\vert \ k_0 \in \{0,1\} \right]  \\
 = & \ \left. \text{min}\left[\text{max}\left\{1+2^{k_0},\left(1+ 2^{k_0} \prod\limits_{ i = 2}^{1+\omega_n-k_0} s_i \right)^{1/n} \right\} \ \right\vert \ k_0 \in \{0,1\} \right] \\
 = & \ \  \text{min}\left[\left(1+ \prod\limits_{ i = 2}^{1+\omega_n} s_i \right)^{1/n},\left(1+ \prod\limits_{ i = 1}^{\omega_n} s_i \right)^{1/n} \right] \\
 = & \ \ \left(1+ \prod\limits_{ i = 1}^{\omega_n} s_i \right)^{1/n}, \end{align*}
which is the trivial bound \eqref{trivialb}. Here $s_i$ denotes the $i^\text{th}$ prime. Instead we apply Theorem \ref{ramifb} and obtain an improvement: Assume first that $2 \nmid x^n-1$. Then we exclude $2$ as a potential factor in the lower bound for $x^n-1$. If instead $2 \mid x^n-1$, then $2^{a+1} \mid x^n-1$. We obtain the bounds
\begin{align*}
x \geq & \ \left(1+ n \cdot \prod\limits_{i=1}^{\omega_n} s_i \right)^{1/n} \quad \text{ for } 2 \mid x^n-1   \\
x \geq & \ \left(1+ \prod\limits_{i=2}^{\omega_n+1} s_i \right)^{1/n}  \quad \ \ \ \text{ for } 2 \nmid x^n-1. 
\end{align*}
Regardless of the divisibility case, we gain an improvement over the trivial bound. We have proved the following Corollary. 
\end{Rem}
\begin{Cor} \label{mini1.3}
Let $n = 2^a$ and $x \in \bb{Z}^+$. Then
$$ x \geq \text{min}\left[\left(1+ n \cdot \prod\limits_{i=1}^{\omega_n} s_i \right)^{1/n},\left(1+ \prod\limits_{i=2}^{\omega_n+1} s_i \right)^{1/n}\right]. $$
\end{Cor}
Note that, in the case when $2 \mid x^n-1$, we can do better. For example, when $n=4$, the proposition above relies on $ 2 \mid x^4-1 $ if and only if $2^3 \mid x^4-1 $ but, as shown by Cohen \cite[Cor. 2.5]{Cohen}, we can prove the stronger statement: $ 2 \mid x^4-1$ if and only if $2^4 \mid x^4-1 $. This follows from the factorisation $x^4-1 = (x-1)(x+1)(x^2+1)$ in $\bb{Z}[x]$ and that $2 \mid x^4-1 $ implies $ 2 \mid x-1$. Then $2 \mid x-1,x+1,x^2+1$ but also $ x+1 \equiv 2 + (x-1) \equiv 0,2 \mod{4}$. Thus $4 \mid x-1$ or $4 \mid x+1$, proving Cohen's result. This principle extends to $n=2^a$ for any $a\in \bb{Z}^+$, resulting in Theorem \ref{modramb}.
\begin{proof}[Proof of Theorem \ref{modramb}]
Consider the $\bb{Z}[x]$-factorisation 
$$ x^n - 1 = (x-1)(x+1) \prod_{i=2}^{a} \left( x^{2^{i-1}}+1 \right).$$
If $2 \mid x$, then we apply the ramification bound as in Proposition 5.2. If instead $2 \nmid x$, then $x \equiv 1 \mod{2}$. Moreover, for any $k \in \bb{Z}^+$, $x^k \equiv 1 \mod{2}$. As before, we argue that $ x+1 \equiv 2+ (x-1) \equiv 0,2 \mod{4}$. Therefore $4 \mid x-1$ or $4 \mid x+1$. Thus, each of the $a+1$ factors of $x^n-1$ are divisible by $2$ and at least one of them is divisible by $4$. Then $2^{a+2} \mid x^n-1$, and the claim follows from the same argument as in Corollary \ref{mini1.3}. \qedhere
\end{proof}
\section{Applications}
In order to demonstrate the utility of these results, we consider two separate applications to existence problems in $\fin{}$ extensions. In what follows, the \textit{hybrid bound} will be taken to mean the bound given in Theorem \ref{hybridb} unless $n = 2^a$, where the \textit{hybrid bound} will instead refer to the specialised version in Theorem \ref{modramb}. In the interest of presenting a cleaner argument and reducing computational complexity, we limit the discussion to the hybrid bounds. Note that the majority of results throughout the chapter can be improved further by applying Theorem \ref{ramifb} instead.
\subsection{The line problem}
Let $q \in \bb{Z}^+$ be a prime power and denote by $\fin{}$ the finite field of order $q$. Consider the following problem. Assume $\gamma_1,\gamma_2 \in \fin{n}$, $\gamma_1,\gamma_2 \neq 0$, such that $\gamma_1/\gamma_2$ generates $\fin{n}$. Does there exist some $a \in \fin{}$ for which $a\gamma_2+\gamma_1$ is primitive in $\fin{n}$? Fix $n \in \bb{Z}^+$ and define
\begin{table}[!ht]
 \begin{tabular}{c c} $\mathcal{L}_n =$ & 
$ \left\{  \begin{tabular}{c|c} $q:=p^m$ & \begin{tabular}{c} $\text{ for all } \gamma_1,\gamma_2 \in \fin{n} \setminus \{0\} \text{ such that } \frac{\gamma_1}{\gamma_2} \text{ generates } \fin{n}, $  \\
   $ \text{ there exists } a \in \fin{} \text{ such that  } a \cdot \gamma_2+\gamma_1 \text{ primitive} $ \end{tabular} \end{tabular} \right\} $.
  \end{tabular}
\end{table} \\
The \textit{line problem} of degree $n$ is to determine which $q \in \mathcal{L}_n$. This problem has been completely solved in the cases where $n=2$ by Cohen \cite[Th. 1.1]{Cohen} and $n=3$ by Bailey, Cohen, Sutherland and Trudgian \cite[Th. 2]{BCST}. Bailey et al. also made significant progress towards the $n=4$ case. For fixed $n$, the standard approach is to take advantage of a sieve to guarantee that all sufficiently large $q$ are members of $\mathcal{L}_n$. In particular, Cohen \cite[Prop. 4.3]{Cohen} proved the following prime sieve criterion.
\begin{Prop}
Let $q$ be a prime power. For any $s$ distinct prime divisors $\varphi_i$ of $q^n-1$, define \label{primesieve} $ \delta = \delta(\varphi_1,...,\varphi_s) := 1 - \sum_{i=1}^s \frac{1}{\varphi_i}$. Suppose that $\delta>0$ for some choice of $\varphi_i$, then
$$ q > R_G:= (n-1)^2 2^{2(\omega_n-s)}\left(\frac{s-1}{\delta}+2\right)^2 \quad \text{ implies } \quad q \in \mathcal{L}_n. $$
\end{Prop}
Thus, we aim to choose $\varphi_i$ such that $R_G$ is minimal. For any fixed $s$, this amounts to maximising $\delta$. Thus, we wish for $\sum_{i=1}^s \varphi_i^{-1}$ to be small. In practice, we choose the largest $s$ distinct prime factors $\varphi_i$ of $q^n-1$. When the explicit factorisation of $q^n-1$ is unknown, we bound $R_G$ above by taking $\varphi_i$ to be the $(\omega_n+1-s)$-th through $\omega_n$-th smallest primes. Minimising across valid $s$ for which $\delta>0$, we obtain an upper bound for $R_G$.
\begin{Rem} \label{sufflarge}
Note that, if we choose $s=0$, then trivially $\delta=1>0$. Thus $ q > (n-1)^2 4^{\omega_n} $ implies $q \in \mathcal{L}_n $. For a fixed $n$, consider the trivial lower bound on $q$ via \eqref{trivialb}. If $ q^n-1 \geq \prod_{i=1}^{\omega_n} s_i$, then $ q > \prod_{i=1}^{\omega_n} s_i > (n-1)^2 4^{\omega_n} $ for all sufficiently large $\omega_n$. In that case, we obtain $q \in \mathcal{L}_n$ automatically. 
\end{Rem}
The difficulty then lies in finding a small upper bound on $\omega_n$ for which $q \not\in \mathcal{L}_n$ is a possibility. Call the largest value of $\omega_n$ for which it is possible that there exists some $q \notin \mathcal{L}_n$ the $\omega_n$ \textit{cut-off} detected by a given lower bound on $q$. Varying across $n \in \{2,30\}$, we compare the effectiveness of the hybrid bound to the trivial bound in detecting $\omega_n$ cut-offs and present these calculations in Table \ref{trivhyb}. 
\begin{table}[ht!]
\scriptsize
\centering 
\caption{The trivial and hybrid bounds compared across varying $n$.}
\label{trivhyb}
\begin{tabular}{|c|c||c|c||c|c|} 
\hline 
 n &\multicolumn{2}{|c|}{$\omega_n \leq \#$}  &  \multicolumn{2}{|c|}{$q \leq \#$} &  search space \\
\hline
& Trivial & Hybrid & Trivial & Hybrid & \\
\rowcolor{lightgray} 2 & 8 & 7 & $3712$  & $2040$ & 55\% \\
\rowcolor{lightgray} 3 & 13 & 13 & $142862$  &  $142862$ &  100\% \\
4 & 18 & 18 & $ 1.2\cdot  10^{6} $ & $ 1.2\cdot  10^{6} $ & 100\% \\
\rowcolor{lightgray} 5 & 23 & 20 & $6.1\cdot  10^{6} $ & $ 3.3\cdot  10^{6} $ & 54\% \\ 
6 & 29 & 28 & $2.8\cdot  10^{7} $ & $ 2.3\cdot  10^{7} $ & 84\% \\ 
\rowcolor{lightgray} 7 & 34 & 26 & $9.0\cdot  10^{7} $ & $ 2.4\cdot  10^{7} $ & 27\% \\ 
8 & 39 & 38 & $2.1\cdot  10^{8} $ & $ 1.9\cdot  10^{8} $ & 90\% \\ 
9 & 44 & 41 & $4.5\cdot  10^{8} $ & $ 3.4\cdot  10^{8} $ & 75\% \\ 
10 & 49 & 43 & $9.0\cdot  10^{8} $ & $ 5.2\cdot  10^{8} $ & 58\% \\ 
\rowcolor{lightgray} 11 & 54 & 39 & $1.7\cdot  10^{9} $ & $ 4.4\cdot  10^{8} $ & 26\% \\ 
12 & 59 & 58 & $3.0\cdot  10^{9} $ & $ 2.8\cdot  10^{9} $ & 92\% \\ 
\rowcolor{lightgray} 13 & 64 & 45 & $5.3\cdot  10^{9} $ & $ 1.1\cdot  10^{9} $ & 21\% \\ 
14 & 69 & 57 & $8.9\cdot  10^{9} $ & $ 3.6\cdot  10^{9} $ & 41\% \\ 
15 & 74 & 66 & $1.5\cdot  10^{10} $ & $ 8.3\cdot  10^{9} $ & 56\% \\ 
16 & 80 & 79 & $2.6\cdot  10^{10} $ & $ 2.4\cdot  10^{10} $ & 93\% \\ 
\rowcolor{lightgray} 17 & 85 & 56 & $4.1\cdot  10^{10} $ & $ 5.1\cdot  10^{9} $ & 12\% \\ 
18 & 90 & 85 & $6.4\cdot  10^{10} $ & $ 4.7\cdot  10^{10} $ & 73\% \\ 
\rowcolor{lightgray} 19 & 95 & 62 & $9.1\cdot  10^{10} $ & $ 1.0\cdot  10^{10} $ & 11\% \\ 
20 & 101 & 91 & $1.3\cdot  10^{11} $ & $ 8.4\cdot  10^{10} $ & 64\% \\ 
21 & 106 & 88 & $1.8\cdot  10^{11} $ & $ 8.0\cdot  10^{10} $ & 44\% \\ 
22 & 111 & 83 & $2.5\cdot  10^{11} $ & $ 6.2\cdot  10^{10} $ & 25\% \\ 
\rowcolor{lightgray} 23 & 116 & 75 & $3.4\cdot  10^{11} $ & $ 3.9\cdot  10^{10} $ & 12\% \\ 
24 & 121 & 120 & $4.5\cdot  10^{11} $ & $ 4.3\cdot  10^{11} $ & 96\% \\ 
25 & 126 & 103 & $6.0\cdot  10^{11} $ & $ 2.3\cdot  10^{11} $ & 38\% \\ 
26 & 132 & 96 & $8.3\cdot  10^{11} $ & $ 1.8\cdot  10^{11} $ & 22\% \\ 
27 & 137 & 124 & $1.1\cdot  10^{12} $ & $ 6.5\cdot  10^{11} $ & 59\% \\ 
28 & 142 & 119 & $1.4\cdot  10^{12} $ & $ 5.7\cdot  10^{11} $ & 40\% \\ 
\rowcolor{lightgray} 29 & 148 & 93 & $2.0\cdot  10^{12} $ & $ 2.0\cdot  10^{11} $ & 10\% \\ 
30 & 153 & 138 & $2.6\cdot  10^{12} $ & $ 1.4 \cdot 10^{12} $ & 55\% \\
\hline
\end{tabular}
\end{table} \\
The \say{search space} column in Table \ref{trivhyb} provides the percentage of potential exceptional $q$ values to check when comparing the hybrid and trivial bounds. The rows for which $n \in \bb{P}$ are highlighted by grey. Computations support the assertion that the hybrid bound is most effective at the primes, when compared to the trivial bound. It is worth noting that even an incremental improvement on the $q$-range will drastically improve computation time. For a specific value $q$, we are essentially factorising $q^n-1$, which becomes increasingly infeasible when $n$ grows large. We consider the specific case when $n=5$. From the above calculation, if $ \omega_5 >20 $ then we must have $ q \in \mathcal{L}_5 $. In general, if $q \not\in \mathcal{L}_5 $ then $q$ is bounded below by the hybrid bound and bounded above by the prime sieve in Proposition \ref{primesieve}. We calculate the $q$-intervals for each $\omega_5 \in [1,20]$, which are presented in Table \ref{hybint}. 
\begin{table}[!ht] 
\scriptsize
\centering
\caption{The hybrid bound intervals.}
\label{hybint}
\begin{tabular}{|c|c|c|}
\hline 
$\omega_5$ &  \# $\leq q$  & $q \leq$ \# \\
\hline
 20 & 2 \ 160 \ 476  &  3 \ 336 \ 768 \\ 
 19 & 881 \ 791  &  2 \ 689 \ 898 \\ 
 18 & 397 \ 068  &  2 \ 150 \ 020 \\
 17 & 132 \ 575  &  1 \ 696 \ 702 \\ 
 16 & 46 \ 411  &  1 \ 327 \ 972 \\ 
 15 & 25 \ 794  &  1 \ 022 \ 189 \\ 
 14 & 9023  &  771 \ 455 \\ 
 13 & 3190  &  571 \ 449 \\
 12 & 1910  &  397 \ 505 \\ 
 11 & 701  &  260 \ 365 \\ 
 10 & 265  &  164 \ 401 \\ 
 9 & 179  &  101 \ 825 \\ 
 8 & 72  &  59 \ 399 \\ 
 7 & 31  &  32 \ 641 \\ 
 6 & 22  &  16 \ 925 \\ 
 5 & 10  &  7834 \\
 4 & 7  &  3175 \\ 
 3 & 4  &  1024 \\
 2 & 3  &  256 \\ 
 1 & 2  &  64 \\
 \hline
\end{tabular} 
\end{table}
Since the upper bound on $q$ is sufficiently small, we enumerate all potential prime powers $q$ in each $\omega_5$-interval. For any such $q$, we determine whether $\omega(q^5-1) = k$. The number of values of $q$ satisfying the $\omega_5$ criterion in each interval is shown in Table \ref{opts}. For $\omega_5 \in \{15,16,17,18,19,20\}$, all values of $q$ are members of $\mathcal{L}_5$.
\begin{table}[!ht] 
\centering
\scriptsize
\caption{$q$-value options in hybrid intervals.}
\label{opts}
\begin{tabular}{|c|c|}
\hline
$\omega_5$ & \# of options \\
\hline
14 &  1 \\
13 &  5 \\
12 &  19 \\
11 &  72 \\
10 &  214 \\
9 &  438 \\
8 &  683 \\
7 &  763 \\
6 &  625 \\
5 &  303 \\
4 &  68 \\
3 &  22 \\
2 &  2 \\
1 &  1 \\
\hline
\text{Total} & 3215 \\
\hline
\end{tabular}
\end{table} \\
In order to whittle down the potential exceptions further, we introduce more sophisticated sieves. In what follows, the radical $\text{Rad}(x)$ of a number $x$ is taken to mean the product of the distinct prime factors of $x$. 
\begin{Prop} (\textit{Modified prime sieve}) Let $q$ a prime power and $\omega_n \geq 2$. Consider $s$ distinct prime factors $\varphi_i$ of $q^n-1$ and some other prime factor $l$. Set $\delta:=\delta(\varphi_1,...,\varphi_s)$ and $k = \text{Rad}(q^n-1)/(l \cdot \prod_{i=1}^s \varphi_i)$. Suppose $\delta > 1/(l \cdot \phi(k))$. Then
$$ q > (n-1)^2 \left( \frac{2^{\omega_n-s-1}\phi(k)(s-1 + 2 \delta) + 1 - \frac{1}{l}}{\phi(k) \delta - \frac{1}{l}}-1 \right)^2 \quad \text{ implies } \quad q \in \mathcal{L}_n. $$
\end{Prop}
For a proof of this fact, see \cite[Prop. 6.3]{Cohen}. In practice, $l$ is always chosen to be the largest prime factor of $q^n-1$ and $k$ to be the product of the smallest factors. Bailey et al. \cite[Lemma 1]{BCST} proved a generalised version.
\begin{Prop} (\textit{General prime sieve}) Let $q$ be a prime power and write the radical of $q^n-1$ as $ \text{Rad}(q^n-1)= k \cdot (\prod_{i=1}^s \varphi_i)( \prod_{j=1}^r l_j )$. Set $\delta = 1-\sum_{i=1}^s\frac{1}{\varphi_i}$, $\epsilon = \sum_{j=1}^r\frac{1}{l_j}$ and $m = \phi(k)/k$. If $\delta m > \epsilon$ then
$$ q > (n-1)^2 \left( \frac{2^{\omega_n-s-r}m(s-1+2\delta) - \delta m+ r - \epsilon}{\delta m-\epsilon}\right)^2  \quad \text{ implies } \quad q \in \mathcal{L}_n. $$
\end{Prop}

Returning to the line problem for $n=5$, we iterate the prime, modified and general sieve to filter out more $q$ values. When applying the sieves, we factorise the radical $ \text{Rad}(q^n-1)= k \cdot (\prod_{i=1}^s\varphi_i)( \prod_{j=1}^r l_j )$ in such a way that the sieve bound is as small as possible. For any $\rho \in \bb{P}$ with $\rho \mid k$ and any $i,j$, we always factorise $\text{Rad}(q^n-1)$ such that $\rho < \varphi_i < l_j$. Then, minimising the sieves amounts to choosing only the number $r$ and $s$ of factors $\varphi_i$ and $l_j$, respectively. The number of unsieved $q$-values for each sieve appear in Table \ref{unsieved}, all values of which are known explicitly.
\begin{table}[!ht] 
\scriptsize
\centering
\caption{Unsieved $q$-values in hybrid intervals.}
\label{unsieved}
\begin{tabular}{|c|c|c|c|}
\hline
$\omega_5$  &  Prime sieve &  \text{Modified sieve} &  \text{General sieve} \\
\hline
14 &  0 &  0 &  0  \\
13 &  0 &  0 &  0 \\
12 &  1 &  1 &  1 \\
11 &  10 &  10 &  10 \\
10 &  25 &  19 &  18 \\
9  &  48 &  37 &  35 \\
8  &  143 &  112 &  111 \\
7  &  189 &  144 &  144 \\
6  &  197 &  161 &  161 \\
5  &  136 &  115 &  115 \\
4  &  42 &  33 &  33 \\
3  &  22 &  16 &  16 \\
2  &  2 &  2 &  2 \\
1  &  1 &  1 &  1 \\
\hline
\text{Totals} &  816 & 651 & 647 \\
\hline
\end{tabular}
\end{table} 
As is apparent, only a small set of prime powers can be non-members of $\mathcal{L}_5$.
\begin{Prop} \label{lineprob}
Let $q$ be a prime power and $E_5$ be the set of 647 prime powers described in the appendix, the largest of which is 62791. If $ q \not\in E_5 $, then $ q \in \mathcal{L}_5 $.
\end{Prop}
\begin{Rem}
Let us generalise our approach for tackling the line problem in the case where $n=5$ to any fixed $n \in \bb{Z}^+$. For some distinguished set $\mathcal{A}_n$, we want to determine for which prime powers $q$ we have the membership $q \in \mathcal{A}_n$. In the case of the line problem, this set $\mathcal{A}_n$ is $\mathcal{L}_n$. We have a sieve criterion of the form
$$ q > S(n,\omega_n) \quad \text{ implies } \quad q \in \mathcal{A}_n, $$
where $S(n,\omega_n)$ is some function of $n$ and $\omega_n$. In Proposition \ref{primesieve}, we might take $s=0$, obtaining the criterion $ q > (n-1)^22^{2 \omega_n} $ implies $q \in \mathcal{A}_n $. Cohen used this sieve in his original work for $\mathcal{L}_3$, see \cite[Cor. 2.4]{Cohen}. This allows us to describe a general procedure for the line problem. A similar approach extends to other existence problems for primitive elements in finite fields. \\

\underline{Step 1:} Finiteness. \\
For fixed $\omega_n$, we compare the sieve criterion $S(n,\omega_n)$ with a lower bound $B(n,\omega_n)$ on $q$, dependant on $n$ and $\omega_n$. Then, if $q \geq B(n,\omega_n) > S(n,\omega_n)$, we conclude that $q \in \mathcal{A}_n$. By Remark \ref{sufflarge}, when $\omega_n$ is sufficiently large, this will always be the case. We obtain a cut-off value $c$ for $\omega_n$, above which it is guaranteed that $q \in \mathcal{A}_n$. The sharpness of the the cut-off will be dependent on the sophistication of both $S(n,\omega_n)$ and $B(n,\omega_n)$. This is where Theorems \ref{hybridb} and \ref{ramifb} become powerful tools for reducing $\omega_n$. \\

\underline{Step 2:} Intervals. \\
For each $\omega_n \leq c$, we construct the intervals $I_{\omega_n} = [B(n,\omega_n),S(n,\omega_n)]$. Thus, if $q \notin \mathcal{A}_n$, then $q \in I_{\omega_n}$ for some value of $\omega_n$. For the line problem, these intervals are described in Table \ref{hybint}. Using Theorem \ref{hybridb} over the trivial bound in \eqref{trivialb}, shortens these intervals by increasing $B(n,\omega_n)$. \\

\underline{Step 3:} Enumeration. \\
If the previous steps have reduced $I_{\omega_n}$ sufficiently, we enumerate the possible prime powers in each interval. For a given prime power in the interval $I_{\omega_n}$, we check if indeed $q^n-1$ has $\omega_n$ distinct prime factors. This is summarised in Table \ref{opts}. \\

\underline{Step 4:} Sieving. \\
For the remaining potential counterexamples $q \notin \mathcal{A}_n$, we reapply our sieve criterion. Since the exact value of $q$ is known, we obtain a better sieve bound in Proposition \ref{primesieve} by including specific information about $q$. In a similar vein, we apply other sieves to eliminate more values of $q$. \\

\underline{Step 5:} Direct verification. \\
Once we exhaust all theoretical sieving techniques, we return to the specific problem of finding a primitive element on a line inside $\fin{n}$. Given that the list of potential counterexamples is small enough to be computable, we verify directly if there exists such primitive elements. 
\end{Rem}
The set $E_5$ of potential counterexamples is expected to contain very few genuine non-members $q \notin \mathcal{L}_5$. In the case of $n=3$, Cohen \cite[Th. 6.4]{Cohen} proved that the number of genuine non-members $q \notin \mathcal{L}_3$ was at most 183. By contrast, Bailey et al. \cite[Th. 2]{BCST} proved that $q \notin \mathcal{L}_3$ if and only if $q \in \{3,4,5,7,9,11,13,31,37\}$. In the same paper, they gave a set of 1469 potential non-members $q \notin \mathcal{L}_4 $ but only 21 of these were proven to be genuine in the above sense. See \cite[Th. 3 \& Th. 5]{BCST}.

\subsection{Primitive element sums}
Let $q \in \bb{Z}^+$ be a prime power and consider the following problem. For which pairs $(q,n)$ does there exist a primitive element $\alpha \in \fin{n}$ such that $\alpha+\alpha^{-1}$ is also primitive? Fix $n$ and define
$$ \mathcal{U}_n := \{ q=p^m \mid \text{ there exists } \alpha \in \fin{n} \text{ such that } \alpha,\alpha+\alpha^{-1} \text{ primitive}\}. $$
Liao, Li and Pu \cite{LLP} provided a partial answer to this question and a sufficient condition for membership $q \in \mathcal{U}_n$.
\begin{Prop}
Let $n \in \bb{Z}^+$ and $q$ a prime power coprime to $n$. Then \label{LLPtheo}
$ q^{n/2} > 2^{2 \omega_n} $ implies $ q \in \mathcal{U}_n. $ 
\end{Prop}
\begin{Rem}
The criterion in Theorem \ref{LLPtheo} is equivalent to \label{shortint}
$$ q > 2^{4\omega_n/n} \quad \text{ implies } \quad q \in \mathcal{U}_n.$$
Thus, for $n > 4\omega_n$, we have $ q \not\in \mathcal{U}_n $ implies $ q \leq 2^{4 \omega_n/n} < 2 $ which in turn results in $ q = 1 $. Since $q$ is assumed to be a prime power, this never occurs. Therefore, for any choice of $\omega_n$, we need only check $n \leq 4 \omega_n$. Given fixed $n$ and $\omega_n$, we have the trivial bound $ q^n > q^n -1 \geq \prod_{i=1}^{\omega_n} s_i  $. Thus, if $q \notin \mathcal{U}_n$, then $q^n$ must lie in the interval $ \prod_{i=1}^{\omega_n} s_i < q^n \leq 2^{4 \omega_n} = 16^{\omega_n} $.
\end{Rem}
\begin{Prop} 
Let $\omega_n \in \bb{Z}^+$, then $ \omega_n < 16 $ if and only if $ \prod_{i=1}^{\omega_n} s_i < 16^{\omega_n}$.
\end{Prop}
\begin{proof}
A quick calculation shows that the right hand side is true whenever $\omega_n < 16$. Moreover, for $\omega_n = 16$, we have $\prod_{i=1}^{16} s_i > 16^{16}$. If $\omega_n = 16+k$ for $k\in \bb{Z}^+$, then
$$ 16^{-\omega_n} \cdot \prod_{i=1}^{\omega_n} s_i = 16^{-k} \prod_{i=17}^{16+k} s_i \cdot \left( 16^{-16} \cdot \prod_{i=1}^{16} s_i \right) > 16^{-k} \prod_{i=17}^{16+k} s_i = \prod_{i=17}^{16+k} \frac{s_i}{16}. $$
Since $s_i > 16$ for any $i>6$, we conclude that indeed $ 16^{-\omega_n} \cdot \prod_{i=1}^{\omega_n} s_i > 1 $. \qedhere
\end{proof}
\begin{Cor}
If $q \notin \mathcal{U}_n $, then $ \omega_n \leq 15$.
\end{Cor}
\begin{proof}
This follows immediately from the fact that $q^n \in (\prod_{i=1}^{\omega_n} s_i,16^{\omega_n}] $ which is non-empty if and only if $\omega_n \leq 15$. \qedhere
\end{proof}
Thus, any possible exceptions $q \notin \mathcal{U}_n$ must satisfy $\omega_n \leq 15$ and, by Remark \ref{shortint}, also $n \leq 4 \omega_n \leq 60$. Note that, for $n=1$, this question has been completely resolved by Cohen, Silva, Sutherland and Trudgian \cite[Cor. 2]{CSST}.
\begin{Prop}
Let $q$ be a prime power. Then
$$q \notin \{2,3,4,5,7,9,13,25,121\} \quad \text{ if and only if } \quad q \in \mathcal{U}_1. $$
\end{Prop}
In what follows, assume that $n > 1$. For fixed values of $\omega_n \in \{1,...,15\}$ we could now check which pairs $(q,n)$ give rise to a value $q^n \in (\prod_{i=1}^{\omega_n} s_i,16^{\omega_n}]$. If instead we apply the hybrid bound $H(n,\omega_n)$ and iterate over $n \in \{1,...,4 \omega_n\}$, we need only check if we have any $q^n \in (H(n,\omega_n)^n,16^{\omega_n}]$. The search intervals for the trivial and hybrid bounds are given in Figures \ref{trivLLP} and \ref{hybLLP}, respectively. Empty search interval pairs $(n,\omega_n)$ are unmarked, whereas potential exceptions are marked black. 
\begin{figure}[!ht]
\centering
\caption{The trivial bound intervals for primitive element sums.}
\includegraphics[width=0.9 \textwidth , height = 0.25 \textwidth]{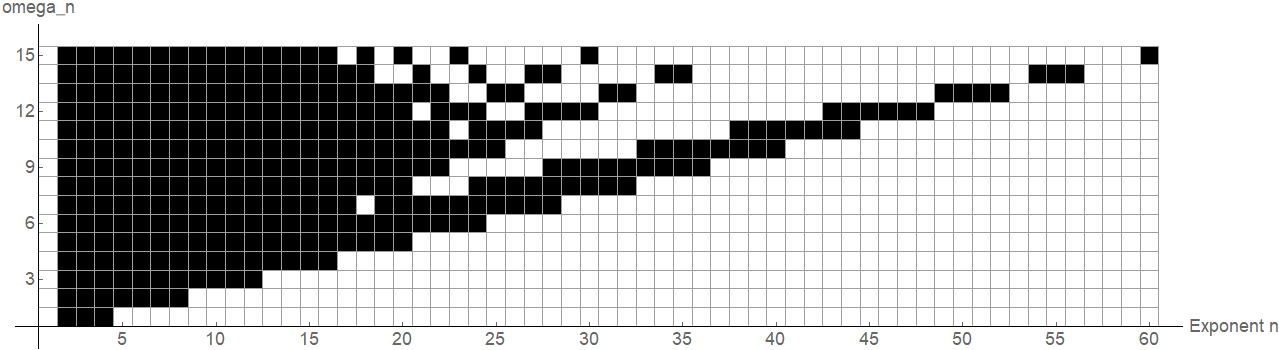}
\label{trivLLP}
\end{figure} 
\begin{figure}[!ht]
\centering
\caption{The hybrid bound intervals for primitive element sums.}
\includegraphics[width=0.9 \textwidth , height = 0.25 \textwidth]{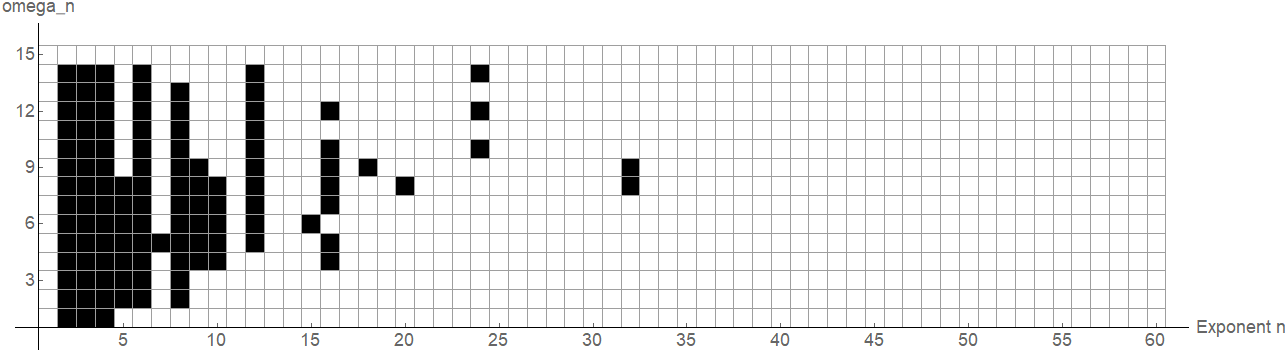}
\label{hybLLP}
\end{figure} \\
The hybrid bound is much more efficient at ruling out potential counterexamples compared to the trivial bound. So far, we have only checked the non-emptiness of each interval $(H(n,\omega_n)^n,16^{\omega_n}]$. The natural next step is to incorporate the presumed value of $\omega_n$ in such an interval. For each pair $(n,\omega_n)$ with non-empty search interval, we check if there exists any prime power $q \in (H(n,\omega_n),16^{\omega_n/n}]$ satisfying $\omega(q^n-1)=\omega_n$. This dramatically reduces our search space and leaves us with only a small list of exceptional pairs $(n,\omega_n)$, marked black in Figure \ref{hybimpLLP}.
\begin{figure}[ht]
\centering
\caption{Existence of $q$-values in hybrid intervals.}
\includegraphics[width=0.9 \textwidth , height = 0.25  \textwidth]{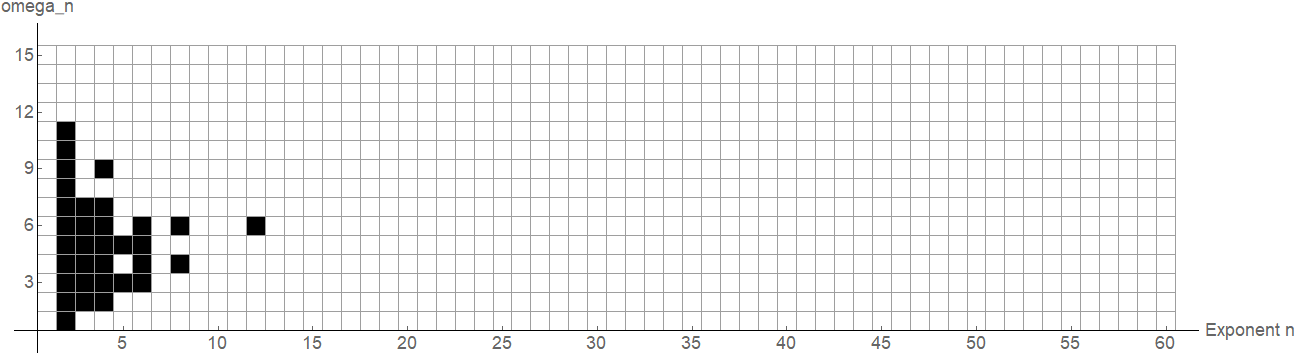}
\label{hybimpLLP}
\end{figure} \\
Our computation gives a sufficient condition for membership of $\mathcal{U}_n$.
\begin{Prop} \label{LLPprob}
Let $q$ be a prime power and $\mathcal{E}_\mathcal{U}$ be the set of 33 pairs $(a,b) \in \bb{Z}^2$ described in the appendix. Then $ (n,\omega_n) \notin \mathcal{E}_\mathcal{U} $ implies $ q \in \mathcal{U}_n $.
\end{Prop}
Although the primitive element sum problem is not completely resolved, the techniques used highlight the relative effectiveness of our hybridised bounds over the trivial bound \eqref{trivialb}. As in Proposition \ref{lineprob}, we should expect that most of the pairs $(n,\omega_n) \in \mathcal{E}_\mathcal{U}$ do not give rise to any genuine non-members $q \notin \mathcal{U}_n$.

\clearpage
\printbibliography[heading = bibintoc, title = Bibliography]

\section*{Appendix}
We describe the set $E_5$ referenced in Proposition \ref{lineprob}: \\
{  $E_5 = \{$ 2, 3, 4, 5, 7, 8, 9, 11, 13, 16, 17, 19, 23, 25, 27, 29, 31, 32, 37, 
41, 43, 47, 49, 53, 59, 61, 64, 67, 71, 73, 79, 81, 83, 89, 97, 101, 
103, 107, 109, 113, 121, 125, 127, 128, 131, 137, 139, 149, 151, 157, 
163, 167, 169, 173, 179, 181, 191, 193, 197, 199, 211, 223, 227, 229, 
233, 239, 241, 243, 251, 256, 257, 269, 271, 277, 281, 283, 289, 293, 
307, 311, 313, 317, 331, 337, 343, 347, 349, 353, 361, 367, 373, 379, 
383, 389, 397, 401, 409, 419, 421, 431, 433, 439, 443, 449, 457, 461, 
463, 467, 487, 491, 499, 509, 512, 521, 523, 529, 541, 547, 569, 571, 
577, 587, 593, 599, 601, 607, 613, 617, 619, 625, 631, 641, 643, 647, 
653, 659, 661, 673, 677, 683, 691, 701, 709, 719, 727, 729, 733, 739, 
743, 751, 757, 761, 769, 773, 787, 797, 811, 821, 823, 827, 829, 839, 
841, 853, 857, 859, 863, 877, 881, 883, 887, 907, 911, 919, 929, 937, 
941, 947, 953, 961, 967, 971, 977, 991, 997, 1009, 1013, 1019, 1021, 
1024, 1031, 1033, 1039, 1051, 1061, 1063, 1069, 1087, 1091, 1093, 
1097, 1103, 1109, 1117, 1123, 1151, 1153, 1163, 1171, 1181, 1193, 
1213, 1217, 1223, 1229, 1231, 1237, 1259, 1277, 1279, 1283, 1289, 
1291, 1301, 1303, 1307, 1321, 1327, 1331, 1361, 1367, 1369, 1373, 
1381, 1399, 1409, 1423, 1429, 1433, 1439, 1447, 1451, 1453, 1471, 
1483, 1489, 1499, 1511, 1531, 1543, 1549, 1567, 1571, 1583, 1597, 
1607, 1609, 1621, 1627, 1637, 1667, 1669, 1681, 1693, 1697, 1699, 
1709, 1721, 1723, 1741, 1747, 1753, 1759, 1777, 1783, 1789, 1801, 
1831, 1849, 1861, 1867, 1871, 1873, 1879, 1901, 1907, 1933, 1949, 
1951, 1993, 1999, 2003, 2011, 2017, 2027, 2029, 2053, 2081, 2083, 
2087, 2089, 2113, 2131, 2137, 2141, 2143, 2161, 2179, 2197, 2203, 
2209, 2221, 2237, 2251, 2267, 2269, 2281, 2293, 2297, 2311, 2333, 
2341, 2347, 2357, 2371, 2377, 2381, 2389, 2393, 2401, 2411, 2437, 
2467, 2473, 2503, 2521, 2531, 2539, 2551, 2557, 2593, 2621, 2633, 
2647, 2659, 2671, 2683, 2689, 2707, 2713, 2731, 2749, 2767, 2791, 
2797, 2803, 2809, 2833, 2851, 2857, 2861, 2887, 2927, 2953, 2963, 
2971, 3001, 3011, 3019, 3023, 3037, 3041, 3061, 3067, 3109, 3121, 
3125, 3181, 3187, 3191, 3217, 3221, 3253, 3259, 3271, 3301, 3319, 
3331, 3361, 3391, 3433, 3457, 3463, 3481, 3511, 3529, 3541, 3547, 
3571, 3583, 3607, 3613, 3631, 3643, 3691, 3697, 3721, 3727, 3733, 
3739, 3767, 3821, 3823, 3851, 3853, 3877, 3907, 3911, 3919, 3931, 
3943, 4003, 4019, 4021, 4051, 4057, 4096, 4111, 4129, 4159, 4201, 
4219, 4229, 4231, 4243, 4261, 4271, 4273, 4327, 4441, 4447, 4489, 
4513, 4519, 4523, 4561, 4567, 4591, 4603, 4621, 4651, 4657, 4663, 
4691, 4733, 4759, 4783, 4789, 4813, 4831, 4909, 4931, 4933, 4951, 
4957, 4999, 5011, 5041, 5101, 5119, 5167, 5179, 5329, 5437, 5449, 
5503, 5521, 5531, 5581, 5591, 5641, 5659, 5701, 5743, 5779, 5791, 
5839, 5851, 5857, 5881, 5923, 6007, 6037, 6043, 6073, 6091, 6121, 
6133, 6163, 6217, 6241, 6247, 6271, 6301, 6361, 6373, 6421, 6427, 
6451, 6571, 6581, 6619, 6637, 6679, 6691, 6733, 6763, 6781, 6791, 
6841, 6859, 6889, 6991, 7039, 7177, 7243, 7351, 7411, 7417, 7549, 
7561, 7573, 7591, 7603, 7687, 7723, 7753, 7921, 7951, 8011, 8089, 
8161, 8191, 8317, 8353, 8419, 8431, 8647, 8737, 8779, 8821, 8893, 
8941, 8971, 9001, 9091, 9109, 9199, 9241, 9283, 9409, 9421, 9521, 
9631, 9661, 9769, 9781, 9871, 10141, 10201, 10321, 10531, 10609, 
10651, 10711, 10831, 10861, 10891, 11047, 11071, 11131, 11257, 11311, 
11317, 11449, 11551, 11701, 11719, 11731, 11881, 11971, 12391, 12433, 
12541, 12769, 13171, 13399, 13567, 13831, 13921, 14281, 14479, 14551, 
14821, 15331, 15511, 15541, 15601, 15625, 15877, 15991, 16339, 16927, 
17161, 17491, 17791, 17851, 18181, 18481, 18769, 19321, 19501, 19891, 
20011, 20431, 20749, 21211, 21751, 22201, 22621, 23011, 23311, 24781, 
25117, 26041, 27691, 28561, 28771, 29671, 30661, 32761, 41611, 42331, 
46411, 51871, 60901, 62791$\}$. } \\ 

\thispagestyle{plain}
We describe the set $\mathcal{E}_\mathcal{U}$ of pairs $(a,b) \in \bb{Z}^2$ referenced in Proposition \ref{LLPprob}: \\
{  $ \mathcal{E}_\mathcal{U} = \{$ (2,1), (2,2), (2,3), (2,4), (2,5), (2,6), (2,7), (2,8), (2,9), (2,10), (2,11), (3,2), (3,3), (3,4), (3,5), (3,6), (3,7), (4,2), (4,3), (4,4), (4,5), (4,6), (4,7), (4,9), (5,3), (5,5), (6,3), (6,4), (6,5), (6,6), (8,4), (8,6), (12,6)$\}. $ } 
\end{document}